\documentclass{amsart}
\usepackage{bbm}
\usepackage{amsxtra,amscd}
\usepackage{graphicx}
\usepackage{amsmath}
\usepackage{amsfonts}
\usepackage{amssymb}

\newtheorem{theorem}{Theorem}[section]

\theoremstyle{definition}
\newtheorem{definition}[theorem]{Definition}
\newtheorem{example}[theorem]{Example}

\newtheorem{proposition}[theorem]{Proposition}
\newtheorem{remark}[theorem]{Remark}

\newtheorem{corollary}[theorem]{Corollary}

\numberwithin{equation}{section}
\begin{document}

\title{Affine Structures on a Ringed Space and Schemes}

\author{Feng-Wen An}
\address{School of Mathematics and Statistics, Wuhan University, Wuhan,
Hubei 430072, People's Republic of China}
\email{fwan@amss.ac.cn}

\begin{abstract}
In this paper we  will first introduce the notion of  affine
structures on a ringed space and then obtain several properties. Affine structures on a ringed space, arising mainly from complex analytical
spaces of algebraic schemes over number fields, behave like differential structures on a smooth manifold.

As one does for differential manifolds, we will use pseudogroups of affine transformations
to define affine atlases on a ringed space. An atlas  on a
space is said to be an affine structure if it is maximal. An affine structure is admissible if there is a sheaf on
the underlying space such that they are coincide on all affine charts, which are in deed affine open sets of a scheme. In a rigour manner, a scheme is defined to be a ringed space with a specified
affine structure if the affine structures are in action in some
special cases such as analytical spaces of algebraic schemes. Particularly, by the whole of affine structures on a space,
we will obtain respectively necessary and sufficient conditions that
two spaces are homeomorphic and that two schemes are isomorphic, which
are the two main theorems of the paper. It follows that the whole of affine structures on a space and a scheme, as local data, encode and reflect the global properties of the space and the scheme, respectively.
\end{abstract}
\maketitle

\section{Introduction}

\subsection{Background and Motivation}

As one studies \emph{differential structures} on a manifold such as Milnor \cite{milnor},   \emph{affine structures} on a scheme, taken as counterparts,  will be introduced and discussed in this paper. Here, we will obtain several properties on affine structures in a rigour and systematic manner. These results  in fact will have been applied to the
discussions on the complex analytical space of an algebraic scheme over a number field.

\subsubsection{An affine covering of a scheme}

As well-known, a scheme (or a projective scheme, respectively) is defined to be a
ringed space that can be covered by a family of affine (or
projective, respectively) schemes, called \emph{an affine covering} of the scheme $X$.  An affine scheme is the spectrum of a
commutative ring equipped with the sheaf in an evident manner (\cite{ega,hartshorne}). In the paper it will be seen that such a family of affine schemes determines a unique \emph{affine
structure} on the scheme.

\subsubsection{Each affine covering produces a complex analytical space}

Fixed an algebraic scheme $X$ over a number field $K$. Let $\{A_{\alpha}\}_{\alpha \in\Gamma}$
be a family of finitely generated algebras over $K$ such that their spectra
$Spec{A_{\alpha}}$ cover $X$, i.e., $$\bigcup_{\alpha} Spec{A_{\alpha}}\supseteq X$$ holds. Here, each $A_{\alpha}$ is isomorphic to the
quotient of some polynomial ring $K[t_{1},t_{2},\cdots,t_{n_{\alpha}}]$  by a
finite number of polynomials $$f_{1},f_{2},\cdots,f_{r_{\alpha}}$$ over $K$ in the variables.

By Serre's
\emph{GAGA} (\cite{sga1,serre-gaga}), each open subscheme $Spec{A_{\alpha}}$ has an analytical space
$X^{an}_{\alpha}$ that is defined by the common zeros of the polynomials
$$f_{1},f_{2},\cdots,f_{r_{\alpha}}$$ mentioned above. Gluing these analytical spaces
$X^{an}_{\alpha}$, we will obtain an analytical space $X^{an}$, called the
\emph{complex analytical space} of $X$. It has been seen that such a process has several functorial
properties with respect to $X$ (\cite{sga1,serre-gaga}).

Hence, every affine covering of $X$ produces a complex analytical space $X^{an}$ of $X$.

However, in general, an algebraic scheme $X$ can have many affine coverings. Then what about the complex analytical spaces of $X$ produced by different affine coverings of $X$?

\subsubsection{Different affine coverings can produce different complex analytical spaces}

Let's take an example raised by Serre (\cite{serre-example}):
\begin{quotation}
\emph{Let $V$ be the
nonsingular projective variety over a number field $K$ as defined in \emph{\cite{serre-example}}. Suppose that
$V_{\phi}$ is a conjugate variety of $V$ defined by an isomorphism $\phi$.
Then there is such an isomorphism $\phi$ that the complex analytic spaces $V^{an}$ and $V^{an}_{\phi}$ are not of
the same homotopy type.}
\end{quotation}

Denote by $A$ and $A_{\phi}$
the homogeneous coordinate rings of $V$ and $V_{\phi}$, respectively. Then
we have projective schemes $$X=Spec{A}; \, X_{\phi}=Spec{A_{\phi}}$$
with complex analytical spaces $$X^{an}=V^{an}; \, X_{\phi}^{an}=V^{an}_{\phi}$$ respectively.

This shows that different affine coverings of the same scheme can produce different complex analytical spaces.

For this phenomenon,  there are also some more examples arising from abelian varieties and Shimura varieties (\cite{deligne-milne,milne-suh}).

Now we come to a conclusion that there do exist evidences, such as related examples in \cite{deligne-milne,sga1,milne-suh,serre-example}, that different affine
coverings can produce different complex analytical spaces for a fixed algebraic scheme.

Why there does exist such a phenomena? It is an interesting problem. Such related topics will be
discussed in our subsequent papers.

\subsubsection{Several problems}

In the above it has been seen that different affine coverings of an algebraic scheme can produce different complex analytical space. Then it is natural for one to have several questions (similarly for projective schemes)
such as the following:
\begin{itemize}
\item    \emph{How can we give a restrictive definition for such a family of affine schemes to
    patch a scheme?}

\item    \emph{Does there exist another family of affine schemes covering
    the fixed scheme and making it into a scheme?}

\item    \emph{Given another family of affine schemes which cover the fixed scheme.
    Will we obtain the same scheme?}

\item    \emph{Given a ringed space. How many families of affine
    schemes do patch it? How many schemes do there exist on the same underlying
    topological space?}

\item    \emph{In particular, given an algebraic scheme $X$ over a number field. There can be
    many families of affine schemes covering $X$. Each such a family produces an analytical
    space $X^{an}$ of $X$. When are these analytical spaces $X^{an}$ either  diffeomorphic to each other or of
    the same homotopy types?}
\end{itemize}

Several questions related to the above will be in part discussed in the paper.

At the same time, affine structures have also been encountered by us during the discussions on a type of Galois covers of algebraic and arithmetic schemes, where such a scheme is said to be \emph{Galois closed} if it has only one affine
structures.

The Galois closed schemes have several nice properties with applications to
class fields, for example, their Galois groups of rational fields are
isomorphic to their groups of automorphisms (for instance, see \cite{an}).

\subsection{Techniques in the Paper}

As a counterpart, an affine structure on a scheme behaves exactly like a differential
structure on a manifold.

\subsubsection{A smooth manifold can have many differential structures}

In a classical way, a differential manifold is a
topological space covered by a family of open subsets in some Euclidean space,
which is obtained by glueing such a family of open sets as patches. Under some
technical conditions, such a family of open sets is called a \emph{differential
structure} on the manifold.

Nowadays there have been many well-known facts about
manifolds and their differential structures:
\begin{itemize}

\item    \emph{There exist differential manifolds which have many
differential structures on them that are not diffeomorphic to each
other} (\cite{milnor}).

\item     \emph{There exist topological spaces that have no differential structures on
them}.

\item     \emph{There exist  differential
manifolds will be of different properties if we establish different differential
structures on the (same) underlying spaces}.
\end{itemize}

\subsubsection{Affine structure v.s. differential structure}

Many approaches and skills in topology can be applied here to schemes. The techniques in the
present paper, which we borrowed from differential topology, is thus not new to
some certain degree.

\subsubsection{Known related results on affine coverings}

For affine coverings, there have been several informal
discussions, for example, see \cite{hartshorne,shafarevich}, on how to patch a scheme, that is, how to glue a given
family of affine schemes into a scheme; for a more abstract case of categories,
there have been fibered categories and groupoids (\cite{sga3-1,sga1}) which can be
applied to discuss such coverings.

However, all those discussions involved in \cite{sga3-1,sga1,hartshorne,shafarevich} deal only with \emph{coverings},
that is, a family of abstract objects over a fixed object.

\subsubsection{Further results obtained in the paper}

In the paper here we will have further discussions on
such affine coverings and several results will be obtained. We will introduce and discuss affine
structures in a rigor and systematic manner, where an affine structure is an affine covering that is taken as maximal families of objects  covering a given object and satisfying the certain properties.

In fact, the affine structures will afford a platform to
us to discuss the problems mentioned above \S 1.1. In particular, an
algebraic scheme over $\mathbb{C}$ can have a unique associative analytical
space if there exists only one affine structure on its underlying space.

Furthermore, in a rigor manner, a scheme is a locally ringed space with a
specified affine structure on it; it follows that in such a case an algebraic
scheme over a number field can be associated exactly with a unique complex
analytical space.

The main results obtained in the paper are that by the whole of affine structures on a space,
it will be seen whether
two spaces are homeomorphic  and whether two schemes are isomorphic. In other words, the whole of affine structures on a space and a scheme, as local data, encode and reflect the global properties of the space and the scheme, respectively.

Such results can be applied to complex analytical spaces of algebraic schemes and arithmetic schemes.

\subsection{Outline of the Paper}

At last we give an outline of the paper.
In \S 2 we will use \emph{pseudogroups $\Gamma$ of affine transformations} to
define an \emph{affine $\Gamma-$atlas} on a topological space, which consists of a
family of \emph{affine charts}. An \emph{affine $\Gamma-$structure} on a space is an affine
$\Gamma-$atlas which is maximal. Our discussion can be regarded as an
algebraic version of differential structures (\cite{str,kobayashi}).

In \S 3 an affine $\Gamma-$structure on a space  is said to be
\emph{admissible} if there is a sheaf on the space such that they are coincide with
each other on each affine chart. Here, such a sheaf is called an \emph{extension} of
the given affine structure.

An affine structure which is not admissible will be
of no practical use.

Given a  scheme $(X,\mathcal{O}_{X})$ in the usual manner (\cite{ega,hartshorne}). In \S 4
we will discuss the special types of affine structures on the space $X$, called
the \emph{canonical} and the \emph{relative canonical} affine structures in the scheme
$(X,\mathcal{O}_{X})$ respectively. Their extensions are called the \emph{associate
schemes} of $(X,\mathcal{O}_{X})$.

Every scheme has an associate scheme. In
particular, a scheme itself is an associate scheme of it. As schemes, a fixed
scheme and their associate schemes are isomorphic with each other.

Now put
\begin{itemize}
\item \emph{$\mathbb{A}\left( X\right)\triangleq$  the set of all admissible affine structures
on a topological space $X$;}

\item \emph{$ \mathbb{A}_{0}\left(
X,\mathcal{O}_{X}\right) \triangleq$  the set of all the relative
canonical affine structures in a scheme} $\left( X,\mathcal{O}_{X}\right)$.
\end{itemize}

Using the set of affine structures on a space, in \S 5 we will give the statements of the two main
theorems in the paper:

\textbf{Theorem 5.1.} \emph{Let $X$ and $Y$ be two
topological spaces such that either $\mathbb{A}\left(
X\right)\not=\emptyset$ or $\mathbb{A}\left( Y\right)\not=\emptyset$
holds. Then $X$ and $Y$ are homeomorphic if and only if there is $$ \mathbb{A}\left(
X\right) =\mathbb{A}\left( Y\right).$$}

\textbf{Theorem 5.2.} \emph{Any two schemes $\left(
X,\mathcal{O}_{X}\right) $ and $ \left( Y, \mathcal{O}_{Y}\right) $ are
isomorphic if and only if we have $$ \mathbb{A}_{0}\left( X,\mathcal{O}_{X}\right)
\cong \mathbb{A}_{0}\left( Y, \mathcal{O}_{Y}\right) .$$}

From the two main theorems it will be seen that the whole of affine structures on a space and a scheme, as local data, encode and reflect the global properties of the space and the scheme, respectively. Such results will be applied to complex analytical spaces of algebraic schemes and arithmetic schemes. The two theorems will be proved in \S 7 and \S 8, respectively.

As a conclusion, in \S 6 we will give several concluding remarks. Particularly, we will come to a conclusion that to be
precisely defined, a scheme should be a locally ringed space together with a
given admissible affine structure on it if the affine structures are in action in a
particular case.

\bigskip

\textbf{Acknowledgment}
The author would like to express his sincere
gratitude to Professor Li Banghe for his advice and instructions
on algebraic geometry and topology.

\section{Definitions for Affine Structures}

In this section we will introduce affine structures on a space in a evident manner as one does for differential structures on a space (for instance, see \cite{str,kobayashi}).

\subsection{Pseudogroup of Affine Transformations}

Let $\mathfrak{Comm}$ be the category of commutative rings with identities,
and $\mathfrak{Comm}/k$ the category of finitely generated algebras over a
field $k.$ Here, a pseudogroup (or groupoid) is a small category in which every
morphism is invertible (\cite{maclane}).

\begin{definition}
A \textbf{pseudogroup $\Gamma$ of affine transformations}, as a subcategory of $\mathfrak{Comm}$, is a pseudogroup of isomorphisms between commutative rings satisfying the
conditions 1-5:
\begin{enumerate}
\item Each $\sigma\in\Gamma$ is an isomorphism from
a ring $dom\left( \sigma\right) $ onto a ring $rang\left( \sigma\right)$
contained in $\Gamma$, called the \textbf{domain} and
\textbf{range} of $\sigma$, respectively.

\item If $\sigma\in\Gamma$,  the inverse
$\sigma^{-1}$ is contained in $\Gamma.$

\item The identity map $id_{A}$ on $A$ is contained in $%
\Gamma$ if there is some $%
\delta\in\Gamma$ with $dom\left( \delta\right) =A.$

\item If $\sigma\in\Gamma$,  the isomorphism
induced by $\sigma$ defined on the localization $dom\left(
\sigma\right)_{f}$ of the ring $dom\left( \sigma\right)$ at any
 $0\not= f\in dom\left( \sigma\right) $ is contained in $\Gamma.$

\item Given any $\sigma,\delta\in\Gamma$. Then
the isomorphism factorized by $dom\left( \tau\right) $ from
$dom\left( \sigma\right) _{f}$ onto $rang\left( \delta\right) _{g}$
is contained in $\Gamma$ if for some
$\tau\in\Gamma$ there are isomorphisms $dom\left( \tau\right) \cong
dom\left( \sigma\right) _{f}$ and $dom\left( \tau\right) \cong
rang\left( \delta\right) _{g}$ with $0\not= f\in dom\left(
\sigma\right) $ and $0\not= g\in rang\left( \delta\right) .$
\end{enumerate}

Such a pseudogroup $\Gamma$ is said to be a \textbf{pseudogroup of } $k-$\textbf{affine transformation} if $\Gamma$
is contained in the category $\mathfrak{Comm}{/k}$, or equivalently,
if each isomorphism in $\Gamma$ is an isomorphism of finitely
generated algebras over a field $k.$
\end{definition}

\subsection{Affine Charts and Affine Atlas}

For a topological space, we give the notions of affine charts and affine atlas.

\begin{definition}
Let $X$ be a topological space and  $\Gamma$ a
pseudogroup of affine transformations. Then an \textbf{affine }$\Gamma-$\textbf{atlas}
 $\mathcal{A}\left( X,\Gamma\right) $ on $X$ is a collection of pairs $\left( U_{j},\varphi_{j}\right)
$ with $j\in\Delta$, called \textbf{affine charts}, satisfying the conditions
1-3:
\begin{enumerate}
\item {For every pair} $\left( U_{j},\varphi
_{j}\right) \in\mathcal{A}\left( X,\Gamma\right) ,${
}$U_{j}${ is an open subset of }$X${ and
}$\varphi_{j}${ is an homeomorphism of }$U_{j}${ onto
}$Spec\left( A_{j}\right) ,${ where }$A_{j}${ is a
commutative ring contained in $\Gamma$.}

\item $\bigcup_{j\in\Delta} U_{j}\supseteq X${ is an open covering
of }$X.$

\item {Given any} $\left( U_{i},\varphi
_{i}\right) ,\left( U_{j},\varphi_{j}\right) \in\mathcal{A}\left(
X,\Gamma\right) $ {with} $U_{i}\cap U_{j}\not
=\varnothing${. There exists a pair }$\left(
W_{ij},\varphi_{ij}\right) \in\mathcal{A}\left( X,\Gamma\right) $
{such that }$W_{ij}\subseteq U_{i}\cap U_{j}${ and that the
isomorphism from the localization }$\left( A_{j}\right)
_{f_{j}}${ onto the localization }$\left( A_{i}\right)
_{f_{i}}${ that is induced by the restriction
}$$\varphi_{j}\circ\varphi_{i}^{-1}\mid_{W_{ij}}:\varphi_{i}(W_{ij})
\rightarrow\varphi_{j}(W_{ij})$$
{is also contained in }$\Gamma$. {Here} $A_{i}$ and $A_{j}$
{are commutative rings contained in $\Gamma$ such that}
$$\varphi _{i}\left( U_{i}\right) =SpecA_{i}\,  and \,
\varphi_{j}\left( U_{j}\right) =SpecA_{j}$$ {hold and that there
are homeomorphisms}
$$\varphi_{i}\left( W_{ij}\right) \cong Spec\left( A_{i}\right)
_{f_{i}}{ and }\varphi_{j}\left( W_{ij}\right) \cong Spec\left(
A_{j}\right) _{f_{j}}$${for some} $f_{i}\in A_{i}$ \, and \, $%
f_{j}\in A_{j}.$
\end{enumerate}

Moreover, $\mathcal{A}%
\left( X,\Gamma\right) ${ is said to be a $k-$\textbf{affine }$\Gamma- $\textbf{atlas}}{ on }$X$ {if
}$\Gamma$ is a subcategory of $\mathfrak{Comm}/k.$

{An affine }$\Gamma-${atlas }$\mathcal{A}\left( X,\Gamma\right) $%
{ on }$X${ is said to be \textbf{complete} (or
\textbf{maximal}) if it can not be contained properly in any other
affine }$\Gamma-${atlas of }$X.$
\end{definition}

\begin{remark}
{The above construction in Definition 2.2 is
well-defined since the open covering }$\{U_{j}\}$ {such that}
$\left( U_{j},\varphi _{j}\right) \in\mathcal{A}\left(
X,\Gamma\right) ${ is a base for the topology on }$X.$

{Let }$\mathcal{A}\left( X,\Gamma\right) $ {%
and }$\mathcal{A}\left( X,\Gamma^{\prime}\right) $ {be atlases
on a
space }$X.$ {Then }$\Gamma\supseteq\Gamma^{\prime}${ holds if }$%
\mathcal{A}\left( X,\Gamma\right) \supseteq\mathcal{A}\left(
X,\Gamma^{\prime}\right) .$
\end{remark}

\subsection{Affine Structures}

As one has differential structures on a manifold, here we have affine structures on a space such as the following.

\begin{definition}
{Let }$X${ be a topological space and }$\Gamma${ a
pseudogroup of affine transformations. Then two affine }$\Gamma-${%
atlases }$\mathcal{A}$ {and} $\mathcal{A}^{\prime}${ on }$X${%
\ are said to be $\Gamma-$\textbf{compatible}} {if the
 condition below is satisfied:}
\begin{quotation}
For any $\left( U,\varphi\right) \in\mathcal{A}$ {and }$%
\left( U^{\prime},\varphi^{\prime}\right) \in\mathcal{A}^{\prime}$
{with} $U\cap U^{\prime}\not =\varnothing$  {there
exists an affine chart }$\left( W,\varphi^{\prime\prime}\right) \in\mathcal{A%
}\bigcap\mathcal{A}^{\prime}$ {such that }$W\subseteq U\cap U^{\prime}$%
{ and that the isomorphism from the localization }$\left( A\right)
_{f}${ onto the localization }$\left( A^{\prime}\right)
_{f^{\prime}}${
induced by the restriction }$\varphi^{\prime}\circ\varphi^{-1}\mid _{W}$%
{ is also contained in }$\Gamma$. {Here }$A$ and $A^{\prime}$
{are commutative rings contained in $\Gamma$ such that}
$\varphi\left( U\right) =SpecA {\, and \, } \varphi^{\prime}\left(
U^{\prime}\right) =SpecA^{\prime}$ {hold and that there are homeomorphisms}
$\varphi\left( W\right) \cong Spec\left( A\right)_{f}{\, and \,
}\varphi^{\prime}\left( W\right) \cong Spec\left( A^{\prime}\right)
_{f^{\prime}}$ {for some} $f\in A$ {and} $f^{\prime}\in
A^{\prime}.$
\end{quotation}
\end{definition}

\begin{proposition}
{Let }$X${ be a topological space and let }$\Gamma${ be a
pseudogroup of affine transformations. Then for any given affine }$\Gamma-$%
{atlas $\mathcal{A}$ on }$X,${ there is a unique complete affine }%
$\Gamma-${atlas }$\mathcal{A}_{m}${ on }$X${ such
that}
\begin{itemize}
\item $\mathcal{A\subseteq A}_{m};$

\item $\mathcal{A}$ {and }$\mathcal{A}_{m}$ {%
are }$\Gamma-${compatible.}
\end{itemize}
\end{proposition}

{In such a case, we will say that $\mathcal{A}$ is a \textbf{base} for
}$\mathcal{A}_{m}$ {and} $\mathcal{A}_{m}${ is the
\textbf{complete affine }$\Gamma-$\textbf{atlas}}{ determined
by $\mathcal{A}.$}

\begin{proof}
Prove the existence. Let $\Sigma$ be the collection of affine $\Gamma-$%
atlases $\mathcal{A}_{\alpha}$ on $X$ such that
$\mathcal{A}\subseteq\mathcal{A}_{\alpha}$ and that
$\mathcal{A}$ and $\mathcal{A}_{\alpha}$ are $\Gamma-$compatible.

Then $\Sigma$ is a partially ordered set together with the inclusions of sets
$\mathcal{%
A}_{\alpha}\subseteq\mathcal{A}_{\beta}$ for any $\mathcal{A}_{\alpha},%
\mathcal{A}_{\beta}\in\Sigma.$ It is clear that every totally
ordered subset of $\Sigma$ has a upper bound in $\Sigma$. By Zorn's
Lemma, $\Sigma$ has maximal elements.

Prove the uniqueness. Let $\mathcal{A}_{m}$ and $\mathcal{A}_{m}^{\prime}$
be two maximal elements of $\Sigma.$ Then we must have $\mathcal{A}_{m}=%
\mathcal{A}_{m}^{\prime}.$ Otherwise, hypothesize $\mathcal{A}_{m}\not =%
\mathcal{A}_{m}^{\prime}.$ It is seen that $\mathcal{A}_{m}$ and $\mathcal{A}%
_{m}^{\prime}$ are $\Gamma-$compatible since they are $\Gamma
-$compatible respectively
with $\mathcal{A}$; then the union $\mathcal{A}_{m}\cup\mathcal{%
A}_{m}^{\prime}$ is contained in $\Sigma,$ where we will obtain a
contradiction.
\end{proof}

\begin{definition}
{Let }$X${ be a topological space. An \textbf{affine }}$%
\Gamma -${\textbf{structure} on }$X$ {is a complete affine }$%
\Gamma -${atlas }$\mathcal{A}\left( \Gamma \right) $ {on
}$X,$ {where $\Gamma $ is a given pseudogroup of affine
transformations.}
\end{definition}

Likewise, we define a {$k-$\textbf{affine
}$\Gamma-$\textbf{structure}} if $\Gamma\subseteq\mathfrak{Comm}/k.$

\section{Admissible Affine Structures}

By Proposition 2.1 it is seen that an affine atlas on a topological space
$X$ determines a unique affine structure on it. From this view of point, we
sometimes identify an affine atlas on $X$ with its determined complete
affine structure on $X.$ In this section will discuss admissible affine structures on a space. On a given space, only admissible affine structures are interesting and are of the practical uses.

\begin{definition}
{Let $\mathcal{A}$}$\left( \Gamma \right) ${ be an affine }$%
\Gamma -${structure on a topological space }$X${. Suppose
that there exists a locally ringed space }$\left(
X,\mathcal{F}\right) $ such that
$\varphi _{\alpha \ast }\mathcal{F}\mid
_{U_{\alpha }}\left( SpecA_{\alpha }\right) =A_{\alpha }$ holds for each $\left( U_{\alpha },\varphi
_{\alpha }\right) \in ${$\mathcal{A}$}$\left( \Gamma \right) $,
{where} $A_{\alpha }$ {is a commutative ring contained in
$\Gamma$ with} $\varphi _{\alpha }\left( U_{\alpha }\right)
=SpecA_{\alpha }.$

{Then }$\mathcal{A}\left( \Gamma
\right) ${ is said to be an \textbf{admissible affine structure} on }$%
X$ {and }$\left( X,\mathcal{F}\right) $ {is said to be an
\textbf{extension}} {of the affine }$\Gamma -${structure }$\mathcal{%
A}\left( \Gamma \right) .$
\end{definition}

\begin{proposition}
{All extensions of an admissible affine
structure on a topological space are schemes which are isomorphic with each other.}
\end{proposition}

\begin{proof}
Let $\mathcal{A}$ be an admissible affine structure on a topological space $X$.
It is evident that each extension of $\mathcal{A}$ on $X$ is a scheme.

Now fixed any extensions $\left( X,\mathcal{F}\right) $ and $
\left( X,\mathcal{G}\right) $  of $\mathcal{A}$ on $X$.
We prove $\mathcal{F}\cong\mathcal{G}.$

In deed, let $U_{\alpha }$ be an open subset of $X$ contained in
$\mathcal{A}$. From the assumption we have $$\Gamma \left(
\mathcal{F},U_{\alpha }\right) =\Gamma \left( \mathcal{G},
U_{\alpha }\right) .$$

Take any open subset $U$ of $X$. We have $$U=\bigcup_{\alpha
}U_{\alpha }$$ with $U_{\alpha }\in \mathcal{A}$.
Define a map $$\phi:\Gamma \left(
\mathcal{F},U\right) =\Gamma \left( \mathcal{G},
U\right),\\ t\mapsto \phi (t)$$
where $\phi (t)\in \Gamma \left( \mathcal{G},
U\right)$ is the section on $U$ determined by $$t\mid _{U_{\alpha}}=
\phi (t)\mid _{U_{\alpha}}.$$
Then $\phi$ is an isomorphism for every open subset $U$ of $X$.

By $\phi$
we obtain an
isomorphism $\mathcal{F}_{x}\cong \mathcal{G}_{x}$ at every $x\in X$ and
hence $\mathcal{F}\cong\mathcal{G}$ holds.
\end{proof}

\begin{corollary}
For affine structures, there are the following statements:
\begin{enumerate}
\item {Let }$\left( X,\mathcal{F}\right) $ {be an
extension of the affine }$\Gamma -${structure }$\mathcal{A}\left(
\Gamma \right) $ {on a space }$X.$ {Then we have}
$$
\left( U,\mathcal{F}\mid _{U}\right) \cong \left( SpecA,\widetilde{A}%
\right) \text{ {and }}\mathcal{F}\mid _{U}=\varphi _{\ast }^{-1}%
\widetilde{A}
$$
{for every affine chart }$\left( U,\varphi \right) \in ${$\mathcal{%
A}$}$\left( \Gamma \right) $ {with }$\varphi \left( U\right)
=Spec\left( A\right) ,${ where }$A$
{is a commutative ring contained in }$%
\Gamma .$

\item {An affine structure }$\mathcal{A}${
on a space }$X${ is admissible if and only if
}$\mathcal{A}${ can be
extended to be a sheaf }$\mathcal{F}${ on }$X$ {such that }$%
\left( X,\mathcal{F}\right) $ {is a locally ringed space.}
\end{enumerate}
\end{corollary}

\begin{proof}
It is immediate from Definition 3.1 and Proposition 3.1.
\end{proof}

\section{Canonical Affine Structures}

In this section it will be seen that for a given scheme there
can be many different admissible affine structures on the underlying
topological space of the scheme. That is, a given scheme can have many associate schemes. All associate schemes of a given scheme are isomorphic as schemes but have different affine structures.

\subsection{Canonical Pseudogroups of Affine Transformations}

To start with, let's  consider an example.

\begin{example}
\textbf{(Different Affine Structures)} {Let} $k$ {be a field.}
\begin{enumerate}
\item {Put} $$\Gamma _{1}=\{\text{{the
identity }}1_{k}:k\rightarrow k\};$$
\begin{equation*}
\begin{array}{c}
\Gamma _{2}=\{\text{{the identity }}1_{k}:k\rightarrow k\}\bigcup \{\text{{a field isomorphism
}}\sigma
:k\rightarrow k^{\prime }\}\\

 \bigcup \{\text{{the inverse }}\sigma
^{-1}:k^{\prime }\rightarrow
k\}.\\
\end{array}
\end{equation*}
{Then }$\Gamma _{1}${ and }$\Gamma _{2}${ are both
pseudogroups of\ affine transformations.}

\item Let
$$\mathcal{A}\left( \Gamma _{1}\right) =\{\left( U,\varphi
\right) \};$$ $$ \mathcal{A}\left( \Gamma _{2}\right)
=\{\left( U,\varphi \right) ,\left( V,\eta \right) \},$$ {where
}$$U=V=Spec(k),\text{ }\varphi \left( U\right) =Spec(k),\text{ and }\eta \left( V\right) =Spec(k^{\prime }).$$
{Then }$Spec(k)${ is an extension of the affine structure }$%
\mathcal{A}\left( \Gamma _{1}\right) $.
{In general, it is not true that
}$Speck${ is an extension of }$\mathcal{A}\left( \Gamma
_{2}\right) .$
{For example, let }$\sqrt[3]{2},\xi ,\overline{\xi }$ {be
the roots of the equation }$X^{3}-2=0$ {in }$\mathbb{C}.$
{Consider }$k=\mathbb{Q}\left( \sqrt[3]{2}\right); k^{\prime } =\mathbb{Q}\left( \xi \right) .$
\end{enumerate}
\end{example}

\begin{definition}
{Let }$\left( X,\mathcal{O}_{X}\right) $ be a scheme.
{Denote by $\Gamma_{ 0}$ (respectively, $\Gamma^{max}$)
the union of
the set of some (respectively, all) identities of commutative rings}
$$id_{A_{\alpha }}:A_{\alpha }\rightarrow A_{\alpha }$$ {and
the set of some (respectively, all) isomorphisms of commutative rings}
$$\sigma _{\alpha \beta }:\left( A_{\alpha }\right) _{f_{\alpha
}}\rightarrow \left( A_{\beta }\right) _{f_{\beta }},$$
{satisfying the conditions 1-2:}
\begin{enumerate}
\item {Each} $ A_{\alpha },A_{\beta },A_{\gamma }\in Comm$
{are commutative rings such that there are affine open subsets}
$$U_{\alpha },U_{\beta }, { and \, }U_{\gamma }\subseteq U_{\alpha
}\cap U_{\beta }$${of} $X$ {satisfying the conditions}
$$\varphi _{\alpha }\left( U_{\alpha }\right) =SpecA_{\alpha
}, \varphi _{\beta }\left( U_{\beta }\right) =SpecA_{\beta
},{ and \,} \varphi _{\gamma }\left( U_{\gamma }\right)
=SpecA_{\gamma }.$$

\item {Each} $\sigma _{\alpha \beta }:\left( A_{\alpha
}\right) _{f_{\alpha }}\rightarrow \left( A_{\beta }\right)
_{f_{\beta }}${is induced from the homeomorphism}
$$\varphi _{\alpha }\circ \varphi _{\beta }^{-1}\mid _{U_{\gamma}}:
\varphi _{\beta }({U_{\gamma}})\rightarrow\varphi _{\alpha}({U_{\gamma}})$$
{such that}
$$
\varphi _{\alpha }\left( U_{\gamma }\right) \cong Spec\left( A_{\alpha
}\right) _{f_{\alpha }}\text{{ and }}\varphi _{\beta }\left( U_{\gamma
}\right) \cong Spec\left( A_{\beta }\right) _{f_{\beta }}
$$
{hold for some} $f_{\alpha }\in A_{\alpha }$ {and
}$f_{\beta }\in A_{\beta }.$
\end{enumerate}

{Then the pseudogroup generated by
$\Gamma_{0}$ in $\mathfrak{Comm}$, denoted by
$\Gamma_{X,\mathcal{O}_{X}}$, which is the smallest pseudogroup
containing  $\Gamma_{0}$ in $\mathfrak{Comm}$, is called \textbf{a
pseudogroup of affine transformations} in $\left(
X,\mathcal{O}_{X}\right) $.}

{The pseudogroup generated by
$\Gamma^{max}$ in $\mathfrak{Comm}$, denoted by
$\Gamma^{max}_{X,\mathcal{O}_{X}}$, is called \textbf{the maximal
pseudogroup of affine transformations} in $\left(
X,\mathcal{O}_{X}\right) $.}

{For any given $\Gamma_{X,\mathcal{O}_{X}}$, define} $$\mathcal{A}^{\ast }
\left( \Gamma_{ X,\mathcal{O}_{X}}
\right) =\{\left( U_{\alpha },\varphi _{\alpha }\right) :\varphi _{\alpha
}\left( U_{\alpha }\right) =SpecA_{\alpha } \text{ and }A_{\alpha }\in
\Gamma_{ X,\mathcal{O}_{X}} \}$${where each $U_{\alpha}$ is an
affine open subset in the scheme $X$.}
\end{definition}

\begin{definition}
{Let }$\left( X,\mathcal{O}_{X}\right) $ be a scheme.
{Given such a pseudogroup $\Gamma_{
X,\mathcal{O}_{X}}$ in $\left(
X,\mathcal{O}_{X}\right)$. Suppose that $\mathcal{A}^{\ast }\left(
\Gamma_{ X,\mathcal{O}_{X}}
\right)$ is an affine $\Gamma_{ X,\mathcal{O}_{X}}
-$atlas on the space $X$.} {Then $\Gamma_{
X,\mathcal{O}_{X}}$ is said to be a \textbf{canonical pseudogroup of affine transformations}
in
the scheme $\left( X,\mathcal{O}_{X}\right)$ and $\mathcal{A}^{\ast }\left(
\Gamma_{ X,\mathcal{O}_{X}}
\right)$ is said to be an \textbf{affine atlas} in the scheme $(X,\mathcal{O}_{X})$.}
\end{definition}

It is immediate that $\Gamma_{X,\mathcal{O}_{X}}$ is a sub-pseudogroup of $\Gamma^{max}_{X,
\mathcal{O}_{X}}$.
There can be many canonical pseudogroups of affine transformations in the scheme
$(X,\mathcal{O}_{X})$. By Zorn's Lemma it is seen that
$\Gamma^{max}_{X,\mathcal{O}_{X}}$ is maximal among these pseudogroups.

Take an example. Let $X=Spec(\mathbb{Z})$ and $Y$ be the disjoint union of $X$.
Then there are three canonical pseudogroups of affine transformations in the
scheme $Y$,
which are generated respectively by $\mathbb{Z}$ and its localisations, by
$\mathbb{Z}\oplus\mathbb{Z}$
and its localisations, and by $\mathbb{Z}$ and $\mathbb{Z}\oplus\mathbb{Z}$
and their localisations.

\subsection{Canonical Affine Structures}

In general, the underlying space of a scheme can have many affine structures on it.

\begin{definition}
{Let $\Gamma$ be a canonical
pseudogroup of affine transformations in a scheme} $\left(
X,\mathcal{O}_{X}\right) .$

{An affine $\Gamma-$atlas $\mathcal{A}$ on the space $X$ is
said to be \textbf{a canonical affine structure} in the scheme $(X,\mathcal{O}_{X})$
if $\mathcal{A}$
is the affine $\Gamma-$structure on $X$ determined by the
affine $\Gamma-$atlas $\mathcal{A}^{\ast }\left( \Gamma \right)$.}

{An affine $\Gamma-$atlas $\mathcal{A}$ on the space $X$
is said to be \textbf{a relative canonical affine structure} in the scheme
$(X,\mathcal{O}_{X})$ if
$\mathcal{A}$ is maximal among all the
affine $\Gamma-$atlases in $(X,\mathcal{O}_{X})$ which contain the
affine $\Gamma-$atlas $\mathcal{A}^{\ast }\left( \Gamma \right)$ and are
$\Gamma-$compatible.}

{A scheme is said to have \textbf{a unique (}respectively,
\textbf{relative}\textbf{)
canonical affine structure} if there exists only one (respectively, relative)
canonical affine
structure in it.}
\end{definition}

\begin{proposition}
{Let }$\Gamma ${ be the maximal pseudogroup of affine
transformations in a scheme} $\left( X,\mathcal{O}_{X}\right) .$ {Then }$%
\mathcal{A}^{\ast }\left( \Gamma \right) ${ is a relative canonical affine $\Gamma
-$structure in} $\left(
X,\mathcal{O}_{X}\right) .$
\end{proposition}

\begin{proof}
Prove $\mathcal{A}^{\ast }\left( \Gamma \right) $ is a $\Gamma -$atlas
of the space $X.$ In fact, it is clear that  $\mathcal{A}^{\ast }\left( \Gamma \right) $ is $\Gamma -$compatible. Hence, it suffices to prove that
$\mathcal{A}^{\ast }\left( \Gamma \right) $ affords us a base for the topology on
the space $X.$

Fixed any point $x\in X$. Take any affine open subset $U\ni x$ in $\left( X,%
\mathcal{O}_{X}\right) $ such that there is an isomorphism $$\left( \varphi ,
\widetilde{\varphi }%
\right) :\left( U,\mathcal{O}_{X}\mid _{U}\right) \cong \left( SpecA,%
\widetilde{A}\right) $$ where $A\in \mathfrak{Comm}.$

There is some $f\in A$ such that $A_{f}$ is contained in $\Gamma$.
Hypothesize that $A_{f}\not\in \Gamma $ holds for any $f\in A.$ Then
$\{\left( U,\varphi \right) \}$and $\mathcal{A}^{\ast }\left( \Gamma \right) $ are
$\Gamma -$compatible; it follows that $\{id_{A}\}\bigcup\Gamma$ is a
pseudogroup of affine transformations in $(X,\mathcal{O}_{X})$,
where $id_{A}:A\rightarrow A$ is the identity map; hence, we have
$$\Gamma \subsetneqq \{id_{A}\}\bigcup\Gamma,$$ which will be in
contradiction with the assumption.

Now take any $f \in A$ such that $A_{f}\in \Gamma $. Let $SpecA$ be
irreducible without loss of generality. We have $$Spec(A_{f})\cong
D(f)\subseteq SpecA.$$ Then $\{\left(
U,\varphi \right) \}$ and $\mathcal{A}^{\ast }\left( \Gamma \right) $ are $\Gamma -$%
compatible.

As $U$ is an affine open subset in $X$, we have $\left(
U,\varphi \right)\in \mathcal{A}^{\ast }\left( \Gamma \right) $; as $\Gamma $ is
maximal in $\left( X,\mathcal{O}_{X}\right) ,$ it is seen that $A$ is
contained in $\Gamma .$

This proves that for any $x\in X$ there is an affine chart $\left(
U,\varphi \right) \in \mathcal{A}^{\ast }\left( \Gamma \right)$ such that
$x\in U$.
\end{proof}

\begin{remark}
Let $\left( X,\mathcal{O}_{X}\right) $ be a scheme. Then
$\mathcal{A}^{\ast }\left( \Gamma_{X,\mathcal{O}_{X}} \right)
\subseteqq \mathcal{A}^{\ast }\left( \Gamma^{m}_{X,\mathcal{O}_{X}}
\right).$
In particular, each relative canonical affine
$\Gamma_{X,\mathcal{O}_{X}}-$structure
in $(X,\mathcal{O}_{X})$ is contained in $\mathcal{A}^{\ast }\left( \Gamma^{m}_{X,
\mathcal{O}_{X}} \right)$.
In general, there can be different relative canonical affine structures in a scheme.
\end{remark}

Furthermore, we have the following conclusions.

\begin{proposition}
Let $\left( X,\mathcal{O}_{X}\right)$ be a scheme. There are the following statements.
\begin{enumerate}
\item {Let $\Gamma$ be a canonical pseudogroup
of affine transformations in $\left( X,\mathcal{O}_{X}\right)$. Then there is a
unique (respectively, relative)
canonical affine $\Gamma-$structure in $\left( X,\mathcal{O}_{X}\right)$.}

\quad Furthermore, given any affine open subset $U$ in $\left( X,\mathcal{O}_{X}\right)$. Then
$U$ is contained in the canonical affine $\Gamma-$structure in $\left( X,\mathcal{O}_{X}\right)$
if and only if $U$ is contained in the relative canonical affine $\Gamma-$structure
in $\left( X,\mathcal{O}_{X}\right)$.

\item {The scheme $\left( X,\mathcal{O}_{X}\right)$ has a unique affine structure
if and only if $\left( X,\mathcal{O}_{X}\right)$
has a unique relative affine structure.}
\end{enumerate}
\end{proposition}

\begin{proof}
$\emph{1}$. It is immediate from definition.

$\emph{2}$. Assume that $\left( X,\mathcal{O}_{X}\right)$ has a unique affine structure.
Hypothesize that $\mathcal{A}(\Gamma _{1})$
and $\mathcal{A}(\Gamma _{2})$ are two distinct relative canonical affine structures
in $\left( X,\mathcal{O}_{X}\right)$ together
with the canonical
pseudogroups $\Gamma_{1}$ and $\Gamma_{2}$ respectively.

From $\Gamma_{1}$ and $\Gamma_{2}$ we obtain two canonical affine structures
$\mathcal{B}(\Gamma _{1})$
and $\mathcal{B}(\Gamma _{2})$ in $(X,\mathcal{O}_{X})$. Then $\mathcal{B}(\Gamma _{1})$
and $\mathcal{B}(\Gamma _{2})$ are neither $\Gamma_{1}-$compatible nor
$\Gamma_{2}-$compatible. Otherwise, if they are $\Gamma_{1}-$compatible, by $(i)$
it will be seen that
$\mathcal{A}(\Gamma _{1})$
and $\mathcal{A}(\Gamma _{2})$ are $\Gamma_{1}-$compatible.

Hence, there are two distinct canonical affine structures in $\left( X,\mathcal{O}_{X}
\right)$,
which is in contradiction to the assumption.

Conversely, assume that $\left( X,\mathcal{O}_{X}\right)$ has a unique relative affine
structure.
If $(X,\mathcal{O}_{X})$ has two distinct canonical affine structures $\mathcal{B}
(\Gamma _{1})$
and $\mathcal{B}(\Gamma _{2})$, by (1) we will obtain two relative
canonical affine structures in $\left( X,\mathcal{O}_{X}\right)$ which are neither
$\Gamma_{1}-$compatible nor
$\Gamma_{2}-$compatible in virtue of the property of a base for the topology on $X$,
where there will be a contradiction.
\end{proof}

\subsection{Associate Schemes}

In the following it will be seen that a scheme can have many associate schemes.

\begin{proposition}
All (respectively,
relative) canonical affine structures in a scheme $\left( X,\mathcal{O}_{X}\right)$ are admissible; moreover,
their extensions are all isomorphic to $\left( X,\mathcal{O}_{X}\right)$ as schemes.
\end{proposition}

\begin{proof}
Let $\mathcal{A}^{\ast }\left( X,\mathcal{O}_{X}\right) $ be a (respectively, relative)
canonical
affine structure on $X.$ Take any $\left( U_{\alpha },\varphi _{\alpha
}\right) \in \mathcal{A}^{\ast }\left( X,\mathcal{O}_{X}\right) .$ There is
the isomorphism
\begin{equation*}
\left( \tau _{\alpha },\widetilde{\tau _{\alpha }}\right) :\left( U_{\alpha
},\mathcal{O}_{X}\mid _{U_{\alpha }}\right) \cong \left( SpecA_{\alpha },%
\widetilde{A_{\alpha }}\right)
\end{equation*}
where $$\varphi _{\alpha }\left( U_{\alpha }\right) =\tau _{\alpha }\left(
U_{\alpha }\right) ;$$ $$\widetilde{\tau _{\alpha }}\left( \widetilde{%
A_{\alpha }}\right) =\tau _{\alpha \ast }\mathcal{O}_{X}\mid _{U_{\alpha }}.$$

This proves that the scheme $(X,\mathcal{O}_{X})$ is at least an extension of
$\mathcal{A}^{\ast }\left( X,\mathcal{O}_{X}\right) $.
It follows that $\mathcal{A}^{\ast }\left( X,\mathcal{O}_{X}\right) $ is admissible.

Take any extension $(X,\mathcal{F})$ of $\mathcal{A}^{\ast }\left( X,\mathcal{O}_{X}\right) $.
By gluing sections, it is seen that $(X,\mathcal{F})$ and $\left( X,\mathcal{O}_{X}\right)$ are
isomorphic schemes.
\end{proof}

\begin{definition}
An \textbf{associate
scheme} of a given scheme $\left( X,\mathcal{O}_{X}\right) $ is an extension on the space $X$ of a
canonical affine structure or a relative canonical affine structure in $\left( X,\mathcal{O}%
_{X}\right) $.
\end{definition}

\begin{remark} By Proposition 4.3 we have the following statements:

$1.$ Every scheme has an associate scheme.
In particular, a scheme  is an associate scheme of itself.

$2.$ All associate schemes of a given scheme are
isomorphic as schemes.
\end{remark}

\section{Statements of the Main Theorems}

In the paper there are two main theorems that will be stated in the following.

\subsection{Definitions and Notations}

Let's fix notations and terminology.

For a topological space $X$,
put
\begin{itemize}
\item $\mathbb{A}\left( X\right)\triangleq$  \emph{the set of all admissible affine structures
on the space} $X$.
\end{itemize}
For a scheme $\left( X,\mathcal{O}_{X}\right)$, set
\begin{itemize}
\item $ \mathbb{A}_{0}\left(
X,\mathcal{O}_{X}\right) \triangleq$  \emph{the set of all the relative
canonical affine structures in the scheme} $\left( X,\mathcal{O}_{X}\right)$.
\end{itemize}

Likewise, we can define $ \mathbb{A}\left( X;k\right)$ and $
\mathbb{A}_{0}\left( X,\mathcal{O}_{X};k\right) $ for $k-$affine
structures.

\begin{definition}
{For two spaces $X$ and $Y$, we say
$$
\mathbb{A}\left( X\right) \subseteq\mathbb{A}\left( Y\right)
$$
if the below condition is satisfied:}
\begin{quotation}
Given any affine chart $\left( U_{\alpha},\varphi_{\alpha}\right) $
contained in an affine structure
$\mathcal{A}\left( X\right) $ belonging to $\mathbb{A}\left( X\right)$. There
is an affine chart $\left( V_{\alpha},\psi_{\alpha}\right) $ contained in an
affine structure $\mathcal{A}\left( Y\right) $ belonging to $\mathbb{A}%
\left( Y\right) $ such that $B_{\alpha}=A_{\alpha}$. Here
$
A_{\alpha},B_{\alpha}\in\mathfrak{Comm}$, $\varphi_{a}\left( U_{\alpha}\right)
=SpecA_{\alpha}$, and $ \psi_{\alpha
}\left( V_{\alpha}\right) =SpecB_{\alpha}.$
\end{quotation}
\end{definition}

\begin{definition}
For two spaces $X$ and $Y$, we say
$$
\mathbb{A}\left( X\right) =\mathbb{A}\left( Y\right)
$$
if there are relations
\begin{equation*}
\mathbb{A}\left( X\right) \subseteq\mathbb{A}\left( Y\right) \text{ and }%
\mathbb{A}\left( X\right) \supseteq\mathbb{A}\left( Y\right).
\end{equation*}
\end{definition}

Likewise, replacing admissible by relative canonical, we define $$
\mathbb{A}_{0}\left( X,\mathcal{O}_{X}\right) =\mathbb{A}_{0}\left( Y,
\mathcal{O}_{Y}\right) $$
for two schemes $\left( X,\mathcal{O}_{X}\right) $ and $\left( Y,\mathcal{O}%
_{Y}\right) .$

\begin{definition}
For two spaces $X$ and $Y$, we say $$ \mathbb{A}\left( X\right)
\trianglelefteq\mathbb{A}\left( Y\right) $$  if the below conditions 1-3 are
satisfied:
\begin{enumerate}
\item \textbf{(Local Isomorphism)} Given any affine chart $\left( U_{\alpha
},\varphi_{\alpha}\right) $ contained in an affine structure $\mathcal{A}
\left( X\right) $ belonging to $\mathbb{A}\left( X\right)$.

\quad Then there is an affine chart
$\left( V_{\alpha},\psi_{\alpha}\right) $ contained in an
affine structure $\mathcal{A}\left( Y\right) $ belonging to $\mathbb{A}
\left( Y\right) $ such that $A_{\alpha}$ and $B_{\alpha}$ are isomorphic
rings. Here $ A_{\alpha},B_{\alpha}\in\mathfrak{Comm}$,
$\varphi_{a}\left( U_{\alpha}\right) =SpecA_{\alpha}$, and $\psi_{\alpha
}\left( V_{\alpha}\right) =SpecB_{\alpha}. $

\item \textbf{(Covering)} Let $\{\left( U_{\alpha },\varphi_{\alpha}\right)
\}_{\alpha\in\Gamma}$ be a family of affine charts $\left( U_{\alpha
},\varphi_{\alpha}\right) $ contained in some affine structures
$\mathcal{A}(\Gamma_{\alpha})$ belonging to $\mathbb{A}\left( X\right)$
such that $\varphi_{a}\left( U_{\alpha}\right)$ $ =SpecA_{\alpha}$ and
$\bigcup_{\alpha\in\Gamma}U_{\alpha }\supseteq X$.

\quad Then
$\bigcup_{i,\alpha}V_{i,\alpha }\supseteq Y$ holds, where $V_{i,\alpha }$
runs through all the affine charts $(V_{i,\alpha },\psi_{i,\alpha})$ contained in
any affine structures $\mathcal{A}(\Gamma_{i,\alpha})$ belonging to $\mathbb{A}%
\left( Y\right) $ such that $B_{i,\alpha }\cong A_{\alpha}$ and
$\psi_{i,\alpha}(B_{i,\alpha })=SpecB_{i,\alpha }$.

\item \textbf{(Filtering)} Let $\left( U_{\alpha },\varphi_{\alpha}\right)$ and
$(U_{\beta},\varphi_{\beta})$ be two affine charts contained in some affine
structures $\mathcal{A}(\Gamma_{\alpha})$ and
$\mathcal{A}(\Gamma_{\beta})$ belonging to $\mathbb{A}\left( X\right)$
respectively. Given any $x_{\alpha}\in SpecA_{\alpha}$ and $x_{\beta}\in
SpecA_{\beta}$ with
$\varphi^{-1}_{\alpha}(x_{\alpha})=\varphi^{-1}_{\beta}(x_{\beta})$, where
$\varphi_{\alpha}\left( U_{\alpha}\right) =SpecA_{\alpha}$ and
$\varphi_{\beta}\left( U_{\beta}\right) =SpecA_{\beta}$.

\quad Then there exist
affine charts $\left( V_{\alpha },\psi_{\alpha}\right)$ and
$(V_{\beta},\psi_{\beta})$ respectively contained in some affine structures
$\mathcal{A}^{\prime}(\Gamma^{\prime}_{\alpha})$ and
$\mathcal{A}^{\prime}(\Gamma^{\prime}_{\beta})$ belonging to
$\mathbb{A}\left( Y\right)$ such that
$\psi^{-1}_{\alpha}\circ\sigma_{\alpha}(x_{\alpha})=\psi^{-1}_{\beta}\circ
\sigma_{\beta}(x_{\beta})$
holds and that there are ring isomorphisms $\delta
_{\alpha}:B_{\alpha}\cong A_{\alpha}\text{ and }\delta_{\beta
}:B_{\beta}\cong A_{\beta},$
 where $\psi_{\alpha}(V_{\alpha})=SpecB_{\alpha}$,  $\psi_{\beta}(V_{\beta})=SpecB_{\beta}$,
 and $\sigma
_{\alpha}:SpecA_{\alpha}\rightarrow SpecB_{\alpha}\text{ and }\sigma_{\beta
}:SpecA_{\beta}\rightarrow SpecB_{\beta}$ are the isomorphisms induced
from $\delta _{\alpha}$ and $\delta_{\beta }$, respectively.
\end{enumerate}
\end{definition}

Likewise, replacing admissible by relative canonical, we define $$
\mathbb{A}_{0}\left( X,\mathcal{O}_{X}\right) \unlhd\mathbb{A}_{0}\left( Y,%
\mathcal{O}_{Y}\right) $$
for two schemes $\left( X,\mathcal{O}_{X}\right) $ and $\left( Y,\mathcal{O}%
_{Y}\right) .$

Such an isomorphism $\delta_{\alpha}:A_{\alpha}\cong B_{\alpha}$ is
called a \emph{deck transformation} from $X$ into $Y.$

\begin{definition}
For two spaces $X$ and $Y$, we say
$$
\mathbb{A}\left( X\right) \cong\mathbb{A}\left( Y\right)
$$
if there are relations $$\mathbb{A}\left( X\right) \trianglelefteq\mathbb{A}\left(
Y\right) \text{ and }\mathbb{A}\left( X\right) \trianglerighteq\mathbb{A}\left(
Y\right) .$$
\end{definition}

Likewise, replacing admissible by relative canonical, we define $$
\mathbb{A}_{0}\left( X,\mathcal{O}_{X}\right) \cong \mathbb{A}_{0}\left( Y,%
\mathcal{O}_{Y}\right) $$
for two schemes $\left( X,\mathcal{O}_{X}\right) $ and $\left( Y,\mathcal{O}%
_{Y}\right) .$

\subsection{Statements of the Main Theorems}

Now we give the statements of the two main theorems of the present paper.

\begin{theorem}
Let $X$ and $Y$ be two topological spaces such
that either $\mathbb{A}\left( X\right)\not=\emptyset$ or $\mathbb{A}\left(
Y\right)\not=\emptyset$ holds. Then $X$ and $Y$ are homeomorphic if and
only if there is $$ \mathbb{A}\left( X\right) =\mathbb{A}\left( Y\right).$$
\end{theorem}

\begin{theorem}
Any two schemes $\left( X,\mathcal{O}_{X}\right)
$ and $ \left( Y, \mathcal{O}_{Y}\right) $ are isomorphic if and only if we have $$
\mathbb{A}_{0}\left( X,\mathcal{O}_{X}\right) \cong \mathbb{A}_{0}\left( Y,
\mathcal{O}_{Y}\right) .$$
\end{theorem}

We will prove Theorems 5.1 and 5.2 in \S 7 and \S 8, respectively.

\begin{remark}
From the two main theorems above it is seen that the whole of affine structures on a space and the underlying space of a scheme, as local data of the space, encode the global data of the space and the scheme, in particular, the global topology of the space and the scheme, respectively.
\end{remark}

\begin{remark}
In Theorem 5.2 the condition $$ \mathbb{A}_{0}\left(
X,\mathcal{O}_{X}\right) \cong \mathbb{A}_{0}\left( Y,
\mathcal{O}_{Y}\right)$$ can not be replaced by $$ \mathbb{A}\left(
X,\mathcal{O}_{X}\right) = \mathbb{A}\left( Y, \mathcal{O}_{Y}\right).$$
For example, consider $X=Spec({\mathbb{Q}})$ and
$Y=Spec({\mathbb{Q}}(\sqrt{2}))$.
\end{remark}

\section{Concluding Remarks}

From Remark 4.2 and Theorem 5.2 we have the following proposition, a comparison between two schemes of the same underlying space.

\begin{proposition}
{Let $\left( X,\mathcal{O}_{X}\right) $ and $ \left(
X,\mathcal{O}^{\prime}_{X}\right)$ be two schemes. The
following statements are equivalent.}
\begin{itemize}
\item $\mathbb{A}_{0}\left( X,\mathcal{O}_{X}\right) \cong
\mathbb{A}_{0}\left(X, \mathcal{O}^{\prime}_{X}\right) $ holds.

\item $\left( X,\mathcal{O}_{X}\right) $ and $ \left(
X,\mathcal{O}^{\prime}_{X}\right)$ are isomorphic schemes.

\item There is an isomorphism $ \left(
X,\mathcal{O}_{X}^{\symbol{94}}\right) \cong \left(
X,\mathcal{O}_{X}^{\prime \symbol{94}}\right)$ for any associate
schemes $\left( X,\mathcal{O}_{X}^{\symbol{94}}\right)$ of $\left(
X,\mathcal{O}_{X}\right) $ and $\left( X,\mathcal{O}_{X}^{\prime
\symbol{94}}\right) $ of and $\left( X,\mathcal{O}_{X}^{\prime
}\right) $.
\end{itemize}
\end{proposition}

\begin{example}
{Let $K/k$ be a Galois extension. Then $SpecK$ has a unique
associate scheme and there exists a unique admissible $k-$affine
structure in the scheme $SpecK$.}
\end{example}

\begin{definition}
{A scheme $\left( X,\mathcal{O}_{X}\right)$ is said to
\textbf{have a property} $P$ \textbf{for an admissible affine
structures} $\mathcal{A}$ on $X$ if as a scheme any extension
$\left( X,\mathcal{O}_{ \mathcal{A}\left( \Gamma \right) }\right)$
of $\mathcal{A}$ has that property $P$.}
\end{definition}

\begin{remark}
Let $\left( X,\mathcal{O}_{\mathcal{X}}\right) $ be a scheme. There are the following conclusions.
\begin{itemize}
\item Fixed an associate scheme
$(X,\mathcal{O}_{\mathcal{A}})$ of
$(X,\mathcal{O}_{X})$. In general, it is not true that $\left( X,\mathcal{O}_{\mathcal{A}%
}\right)=(X,\mathcal{O}_{X}) $ although they are isomorphic.

\item {There can be a scheme }$\left( X,\mathcal{O}_{X}\right) $
 and an admissible affine structure $\mathcal{A}$ on
the space $X$ such that there is some property $P$
that $\left(
X,\mathcal{O}_{X}\right) $ holds but an extension $\left( X,\mathcal{O}_{\mathcal{A}}\right) $ of $\mathcal{A}$ does not hold.

\item One says that a scheme $\left( X,\mathcal{O}
_{X}\right) $ has a property $P$. But it is not specified that the
property $P$ holds for some certain or all the admissible affine
structures on the space $X$.

\quad This situation is very similar to that in differential
topology. As usual, a differential manifold is said to have some
property if the property holds for all the differential structures
until such a structure is especially specified.

\quad It has been known that
there is some property $P$ on some manifold $X$ which does not hold for any
other differential structures on $X$.
\end{itemize}
\end{remark}

\begin{remark}
In a precise and rigour manner, a scheme is defined to be a ringed space
together with a specified admissible affine structure on it if the affine structures
are in action in a particular case.
\end{remark}

\begin{remark}
{Let $\left( X,\mathcal{O}_{X}\right) $ be a
scheme. For any $\mathcal{A},\mathcal{B}\in \mathbb{A}%
\left( X\right) $, we say $$\mathcal{A}\sim \mathcal{B}$$ if and only
if there is an isomorphism $\left(
X,\mathcal{O}_{\mathcal{A}}\right) \cong \left(
X,\mathcal{O}_{\mathcal{B}}\right) .$}

{Then the quotient set $$\mathbb{A}%
\left( X\right) /\sim $$ is the whole of the schemes on the space $X$
upon isomorphisms.}
\end{remark}

\section{Proof of the First Main Theorem}

In this section we give the proof of Theorems 5.1.

\begin{proof}
\textbf{(Proof of Theorem 5.1)}
As $\mathbb{A}\left( X\right)\not=\emptyset$ and $\mathbb{A}\left(
Y\right)\not=\emptyset$, we can choose two admissible affine structures
on $X$ and $Y$ respectively. Fixed their extensions $\left(
X,\mathcal{O}_{X}\right) $ and $\left( Y, \mathcal{O}_{Y}\right)$, which are
schemes.

$\implies$. Prove that $\mathbb{A}\left( X\right) \subseteq
\mathbb{A}\left( Y\right)$ holds.

In deed, take any affine chart $\left( U_{\alpha },\varphi _{\alpha }\right)
$ contained in an affine structure $\mathcal{A}\left( X\right) $
belonging to $\mathbb{A}\left( X\right) ,$ where $\varphi _{a}\left(
U_{\alpha }\right) =SpecA_{\alpha }$ and $A_{\alpha }\in
\mathfrak{Comm}.$

Then $\left( \tau \left( U\right) ,\varphi _{\alpha }\circ
\tau ^{-1}\right) $ is an affine chart contained in an affine structure $$
\mathcal{A}\left( \tau \left( X\right) \right) =\mathcal{A}\left( Y\right) $$
belonging to $\mathbb{A}\left( Y\right) .$ Hence, we have
$$\mathbb{A}\left( X\right) \subseteq \mathbb{A}\left( Y\right) .$$

Similarly, we have $$\mathbb{A}\left( X\right) \supseteq \mathbb{A}\left(
Y\right).$$ So $$\mathbb{ A}\left( X\right) = \mathbb{A}\left( Y\right).$$

$\impliedby$. Let $\mathbb{A}\left( X\right) =\mathbb{A}\left( Y\right)
$. We will prove that there exists a homeomorphism
$$\tau:X\longrightarrow Y.$$ We will proceed in several steps.

$(i)$ Take an affine chart $\left( U_{\alpha},\varphi_{\alpha}\right) $ contained
in an affine structure
$\mathcal{A}\left( X\right) $ belonging to $\mathbb{A}\left( X\right)$, where
$
A_{\alpha}\in\mathfrak{Comm}$ is a commutative ring and $$\varphi_{a}\left( U_{\alpha}
\right) =SpecA_{\alpha}
.$$ Then there
is an affine chart $\left( V_{\alpha},\psi_{\alpha}\right) $ contained in an
affine structure $\mathcal{A}\left( Y\right) $ belonging to $\mathbb{A}%
\left( Y\right) $ such that $$ \psi_{\alpha
}\left( V_{\alpha}\right) =SpecA_{\alpha}.$$

The converse is true since we have $$\mathbb{A}\left( X\right) =\mathbb{A}\left( Y\right)
.$$

Let $\Sigma$ be the disjoint union of all such open sets $SpecA_{\alpha}$. Take any
points $x,y\in\Sigma$.

We say $$x\thicksim_{X} y$$ if there exist admissible affine
structures $\mathcal{A}(\Gamma_{\alpha})$ and $\mathcal{A}(\Gamma_{\beta})$
contained in $\mathbb{A}\left( X\right)$ satisfying the condition:
\begin{quotation}
\emph{There are affine charts $(U_{\alpha},\varphi _{\alpha })\in
\mathcal{A}(\Gamma_{\alpha}),(U_{\beta},\varphi _{\beta })\in
\mathcal{A}(\Gamma_{\beta})$ such that $\varphi _{\alpha}^{-1}(x)=
\varphi _{\beta }^{-1}(y)$, where $x\in SpecA_{\alpha}=
\varphi_{\alpha}(U_{\alpha})$ and $y\in
SpecA_{\beta}=\varphi_{\beta}(U_{\beta}).$}
\end{quotation}

Likewise, we say $$x\thicksim_{Y} y$$ if there exist admissible affine
structures $\mathcal{A}(\Gamma_{\alpha})$ and $\mathcal{A}(\Gamma_{\beta})$
contained in $\mathbb{A}\left( Y\right)$ satisfying the condition:
\begin{quotation}
\emph{There are affine charts $(V_{\alpha},\psi _{\alpha })\in
\mathcal{A}(\Gamma_{\alpha}),(V_{\beta},\psi_{\beta })\in
\mathcal{A}(\Gamma_{\beta})$ such that $\psi _{\alpha}^{-1}(x)= \psi
_{\beta }^{-1}(y)$, where $x\in SpecA_{\alpha}=
\psi_{\alpha}(V_{\alpha})$ and $y\in
SpecA_{\beta}=\psi_{\beta}(V_{\beta}).$}
\end{quotation}

$(ii)$ Let $\Sigma_{X}$ be the quotient of the set $\Sigma$ by relation $\thicksim_{X}$,
and let
$$\pi_{X}:\Sigma\longrightarrow \Sigma_{X}$$ be the canonical map.

Prove that there is a bijection $\rho_{X}$ from the set $\Sigma_{X}$ onto the set $X$.

In deed, we have a mapping
$$\rho:\Sigma \longrightarrow X$$ given by $$z\longmapsto \varphi_{\alpha}^{-1}(z)$$ where
$(U_{\alpha},\varphi
_{\alpha })$ is the affine chart contained in an affine structure belonging to
$\mathbb{A}(X)$ such that
$$z\in SpecA_{\alpha}=\varphi_{\alpha }(U_{\alpha}).$$ Then we have a map$$\rho_{X}:\Sigma_{X}
\longrightarrow X,
\pi_{X}(z)\longmapsto \rho (z).$$

Evidently, $\rho_{X}$ is a surjection. From the definition
for $\thicksim_{X}$, it is easily seen
that $\rho_{X}$ is an injection. This proves $\rho_{X}$ is a bijection.

Hence, $\Sigma_{X}$ is a topological space together with the topology on the space $X$.

Similarly, let $\Sigma_{Y}$ be the quotient of the set $\Sigma$ by relation $\thicksim_{Y}$,
and let
$$\pi_{Y}:\Sigma\longrightarrow \Sigma_{Y}$$ be the canonical map.
As $$\mathbb{A}\left( X\right) =\mathbb{A}\left( Y\right),$$
it is clear that there is a bijection $\rho_{Y}$ from the set $\Sigma_{Y}$ onto the set $Y$.
Then
$\Sigma_{Y}$ is a topological space together with the topology on the space $Y$.

$(iii)$ Take any $x,y\in\Sigma$. Prove that $$x\thicksim_{X} y$$ holds if and only if
$$x\thicksim_{Y} y$$ holds.

In deed, let $x\thicksim_{X} y$, that is, we have $$\varphi
_{\alpha}^{-1}(x)= \varphi _{\beta }^{-1}(y)$$ for some affine charts
$$(U_{\alpha},\varphi _{\alpha })\in
\mathcal{A}(\Gamma_{\alpha})\text{ and }(U_{\beta},\varphi _{\beta })\in
\mathcal{A}(\Gamma_{\beta})$$ such that $$x\in
SpecA_{\alpha}=\varphi_{\alpha}(U_{\alpha})\text{ and }y\in
SpecA_{\beta}=\varphi_{\beta}(U_{\beta}).$$

Let $\Gamma^{max}_{X,\mathcal{O}_{X}}$be the maximal
pseudogroup of affine transformations in the scheme $\left(X,\mathcal{O}_{X}\right) $.

We choose the above open sets
$U_{\alpha}$ and $U_{\beta}$ to be affine open subsets of the scheme $\left(X,
\mathcal{O}_{X}\right) $. That is, $U_{\alpha}$ and $U_{\beta}$
are contained in the pseudogroup $\Gamma^{max}_{X,\mathcal{O}_{X}}$.

Then there is an affine open subset $U_{\alpha \beta}$ of $\left(X,
\mathcal{O}_{X}\right) $ contained in
$\Gamma^{max}_{X,\mathcal{O}_{X}}$
such that $$\varphi _{\alpha}^{-1}(x)\in U_{\alpha \beta} \subseteq U_{\alpha}
\bigcap U_{\beta};$$
 $$\varphi_{\alpha}(U_{\alpha \beta})=Spec(A_{\alpha})_{f_{\alpha}};$$
 $$\varphi_{\beta}
(U_{\alpha \beta})=Spec(A_{\beta})_{f_{\beta}}$$
for some $f_{\alpha}\in A_{\alpha}$ and $f_{\beta} \in A_{\beta}$, where the isomorphism
$\sigma_{\alpha \beta}$ from
$(A_{\alpha})_{f_{\alpha}}$ onto $(A_{\beta})_{f_{\beta}}$ is contained in
$\Gamma^{max}_{X,\mathcal{O}_{X}}$.

As $$\mathbb{A}\left( X\right) =\mathbb{A}\left( Y\right),$$ we have affine charts
$(V_{\alpha},\psi_{\alpha})$
and $(V_{\beta},\psi_{\beta})$ respectively contained in some affine structures
belonging to $\mathbb{A}(Y)$, where
$$V_{\alpha}=\psi_{\alpha}^{-1}(SpecA_{\alpha})\text{ and
}V_{\beta}=\psi_{\beta}^{-1}(SpecA_{\beta}).$$

Set $$V_{\alpha \beta}=\psi_{\alpha}^{-1}(Spec(A_{\alpha})_{f_{\alpha}});$$
$$V_{\beta \alpha}=\psi_{\beta}^{-1}(Spec(A_{\beta})_{f_{\beta}}).$$
Denote by $\psi_{\beta \alpha}$ the homeomorphism of $V_{\beta \alpha}$ onto
$V_{\alpha \beta}$ which is induced
from $\sigma_{\alpha \beta}$.

It is easily seen that $(V_{\alpha \beta},\psi_{\beta}\circ\psi_{\beta \alpha}^{-1})$
is an affine chart contained in
some admissible affine structure belonging to $\mathbb{A}(Y)$. In fact, fixed any
admissible affine $\Gamma_{0}-$structure $\mathcal{A}(\Gamma_{0})$ on the space $Y$ which
contains the affine chart $(V_{\alpha},\psi_{\alpha})$. Let $\Gamma_{1}$ be the pseudogroup
of affine transformations
in $\frak{comm}$ generated by the union of $\Gamma_{0}$ and the set of the identity on
$(A_{\beta})_{f_{\beta}}$ and all
the possible isomorphisms between the localisations of the rings. Then
$\{(V_{\alpha \beta},\psi_{\beta}\circ\psi_{\beta \alpha}^{-1})\}$
and $\mathcal{A}(\Gamma_{1})$ are $\Gamma_{1}-$compatible. Hence,
$$(V_{\alpha \beta},\psi_{\beta}\circ\psi_{\beta \alpha}^{-1})\in \mathcal{A}(\Gamma_{1}).$$

Consider $$y\in V_{\beta
\alpha};$$ $$x\in Spec(A_{\alpha})_{f_{\alpha}}\subseteq SpecA_{\alpha}.$$ It is evident that $$\psi_{\alpha}^{-1}(x)=
\psi_{\beta \alpha}\circ \psi_{\beta}^{-1}(y)$$ holds since $$\psi_{\alpha}^{-1}
(Spec(A_{\alpha})_{f_{\alpha}})
=\psi_{\beta \alpha}\circ \psi_{\beta}^{-1}(Spec(A_{\beta})_{f_{\beta}})=V_{\alpha \beta}.$$
This proves $$x\thicksim_{Y}y$$ holds.

In a similar manner, it is seen that the converse is true.

$(iv)$ The map from $\Sigma_{X}$ into $\Sigma_{Y}$ defined by $$\pi_{X}(z)\longmapsto
\pi_{Y}(z)$$ for $z\in \Sigma$
gives us a bijection $$\tau:X\longrightarrow Y,$$ which is well-defined from $(iii)$.

All the open sets $SpecA_{\alpha}$ determine a topology on the set
$\Sigma$ in such a manner:
\begin{quotation}
\emph{A subset $W$ of $\Sigma$ is open if and only if $\pi_{X}(W)$ is open
in $\Sigma_{X}$.}
\end{quotation}

It follows that $\Sigma_{X}$ is the quotient space of $\Sigma$ by $\pi_{X}$.
As $\mathbb{A}\left( X\right) =\mathbb{A}\left( Y\right)$, $\Sigma_{Y}$ is
the quotient space of $\Sigma$ by $\pi_{Y}$. Hence, $$\tau:X\longrightarrow
Y$$ is a homeomorphism.

This completes the proof.
\end{proof}

\section{Proof of the Second Main Theorem}

In this section we give the proof of Theorem 5.2.

\begin{proof}
\textbf{(Proof of Theorem 5.2)}
$\implies$. Let $$\tau:\left( X,\mathcal{O}_{X}\right) \cong \left( Y,
\mathcal{O}_{Y}\right)$$ be an isomorphism.

As $$\tau _{\ast }\mathcal{O}_{X}\cong \mathcal{O}_{Y},$$ we have
\begin{equation*}
\begin{array}{l}
\left( \varphi _{\alpha }^{-1}\left( SpecA_{\alpha }\right) ,\left(
\varphi _{\alpha }^{-1}\right) _{\ast }\widetilde{A_{\alpha }}\right)\\

\cong\left( U_{\alpha },\mathcal{O}_{X}\mid _{U_{\alpha }}\right) \\

\cong \left( \tau \left( U_{\alpha }\right) ,\tau _{\ast }\mathcal{O}
_{X}\mid _{U_{\alpha }}\right) \\

\cong \left( \tau \left( U_{\alpha }\right) ,\mathcal{O}_{Y}\mid
_{\tau \left( U_{\alpha }\right) }\right) \\

\cong\left( \psi _{\alpha }^{-1}\left( SpecB_{\alpha }\right) ,\left( \psi
_{\alpha }^{-1}\right) _{\ast }\widetilde{B_{\alpha }}\right)
\end{array}
\end{equation*}
for any affine open set $U_{\alpha}$ of $X$ such that $$\varphi _{\alpha }
\left( U_{\alpha }\right) =SpecA_{\alpha }$$ and $\left( U_{\alpha },\varphi
_{\alpha }\right) $ is contained in some canonical affine structure
belonging
to $\mathbb{%
A}_{0}\left( X,\mathcal{O}_{X}\right)$.

Then $\left( \tau \left( U_{\alpha }\right) ,\varphi _{\alpha }\right) $ is an affine
chart with $$\psi _{\alpha }\left( \tau \left( U_{\alpha }\right) \right)
=SpecB_{\alpha },$$ which is contained in some canonical affine structure
belonging to $\mathbb{A}_{0}\left( Y,\mathcal{O}_{Y}\right) $.

By the
isomorphism $\tau$ it is easily seen that the conditions $\emph{1}$-$\emph{2}$ in
Definition 5.3 are satisfied.

Now let $\Gamma_{X}$ and $\Gamma_{Y}$ be the maximal pseudogroups
of affine transformations in the schemes $\left( X,\mathcal{O}_{X}\right)$
and $\left( Y,\mathcal{O}_{Y}\right)$, respectively. Via the isomorphism
$\tau$, every affine chart in $\mathcal{A}^{*}(\Gamma_{X})$ is an affine chart
in the $\mathcal{A}^{*}(\Gamma_{Y})$; the converse is true.

Using the same
procedure in proving Theorem 5.1, we can prove such a claim that
\textquotedblleft $x\thicksim _{X}y$\textquotedblright\ for $X$ holds if and
only if \textquotedblleft $x\thicksim _{Y}y$\textquotedblright\ for $Y$
holds.

It follows that the condition $\emph{3}$ in Definition 5.3 is satisfied. Hence, we
have
$$\mathbb{A}_{0}\left( X,\mathcal{O}_{X}\right) \trianglelefteq \mathbb{A}%
_{0}\left( Y,\mathcal{O}_{Y}\right).$$

Similarly, we prove $$
\mathbb{A}_{0}\left( X,\mathcal{O}_{X}\right) \trianglerighteq \mathbb{A}%
_{0}\left( Y,\mathcal{O}_{Y}\right).$$ This proves $$\mathbb{A}_{0}
\left( X,\mathcal{O}_{X}\right) \cong \mathbb{A}_{0}\left( Y,%
\mathcal{O}_{Y}\right).$$

$\impliedby$. Put $$\mathbb{A}_{0}\left( X,\mathcal{O}_{X}\right) \cong \mathbb{A}%
_{0}\left( Y,\mathcal{O}_{Y}\right) .$$ We prove that there exists an
isomorphism from $\left( X,\mathcal{O}_{X}\right) $ onto $\left(
Y,\mathcal{O}_{Y}\right) $.

In the following we will proceed in several steps in a manner similar to the procedure in proving
Theorem 5.1.

$(i)$ Let $\Gamma_{Y}$ be the maximal pseudogroup of affine
transformations in the scheme $\left(Y,\mathcal{O}_{Y}\right) $. We obtain a
relative canonical affine $\Gamma_{Y}-$structure $$\mathcal{A}\left(
\Gamma_{Y}\right)\triangleq\mathcal{A}^{*}\left( \Gamma_{Y}\right) $$ on
$Y$. Then $V_{\alpha}$ is an affine open set in the scheme
$(Y,\mathcal{O}_{Y})$ for every $(V_{\alpha},\psi_{\alpha})$ contained in $\mathcal{A}\left(
\Gamma_{Y}\right)$.

For each $(V_{\alpha},\psi_{\alpha})\in \mathcal{A}\left( \Gamma_{Y}\right) $
we put
$$
\psi_{\alpha}(V_{\alpha})=SpecB_{\alpha}
$$
where $B_{\alpha}$ is a commutative ring contained in the pseudogroup
$\Gamma_{Y}$.

From Definition 4.1 we have $$\mathcal{A}\left( \Gamma_{Y}\right) \supseteq
\mathcal{A}^{*}\left( \Gamma_{Y}\right);$$ then
$$\bigcup_{(V_{\alpha},\psi_{\alpha})\in \mathcal{A}\left( \Gamma_{Y}\right) }
V_{\alpha} \supseteq Y.$$

Let $\Sigma$ be the disjoint union of all the open sets $SpecB_{\alpha}$
such that $$\psi_{\alpha}(V_{\alpha})=SpecB_{\alpha}\text{ and
}(V_{\alpha},\psi_{\alpha})\in \mathcal{A}\left( \Gamma_{Y}\right) .$$ Take
any points $x,y\in\Sigma$.

We say $$x\thicksim_{Y} y$$ if there are affine charts $(V_{\alpha},\psi _{\alpha
}),(V_{\beta},\psi_{\beta })\in \mathcal{A}\left( \Gamma_{Y}\right) $ such that
$$\psi _{\alpha}^{-1}(x)= \psi _{\beta }^{-1}(y),$$ where $$x\in SpecB_{\alpha}=
\psi_{\alpha}(V_{\alpha});$$ $$y\in SpecB_{\beta}=\psi_{\beta}(V_{\beta}).$$

$(ii)$ For each $(V_{\alpha},\psi_{\alpha})\in \mathcal{A}\left(
\Gamma_{Y}\right) $, define
$$\{(U_{i,\alpha},\varphi_{i,\alpha})\}_{i\in I_{\alpha}}$$ to be the
set of all the affine charts contained in each relative canonical
affine structures in the scheme $(X,\mathcal{O}_{X})$ such that
$$\varphi_{i,\alpha}(U_{i,\alpha})=SpecA_{i,\alpha}$$ and that there
is an isomorphism $$\delta_{i,\alpha}: A_{i,\alpha}\cong B_{\alpha}.$$

Denote by $\Delta_{X}$ the set of all such affine charts
$(U_{i,\alpha},\varphi_{i,\alpha})$, where $i\in I_{\alpha}$ and
$(V_{\alpha},\psi_{\alpha})\in \mathcal{A}\left( \Gamma_{Y}\right) $.

As $$\mathbb{A}_{0}\left( Y,\mathcal{O}_{Y}\right) \unlhd\mathbb{A}_{0}\left(
X, \mathcal{O}_{X}\right) ,$$ we have
$$\bigcup_{(U_{i,\alpha},\varphi_{i,\alpha})\in \Delta_{X}} U_{i,\alpha} \supseteq X.$$

Let $\Sigma^{*}$ be the disjoint union of all the open sets
$SpecA_{i,\alpha}$ such that $$A_{i,\alpha}\cong B_{\alpha}\text{
and }(U_{i,\alpha},\varphi_{i,\alpha})\in \Delta_{X}.$$ Take any
$x,y\in \Sigma^{*}$.

We say $$x\thicksim_{X} y$$ if there are affine charts
$(U_{i,\alpha},\varphi _{i,\alpha }),(U_{j,\beta},\varphi _{j,\beta
})\in \Delta_{X}$ such that $$\varphi _{i,\alpha}^{-1}(x)= \varphi
_{j,\beta }^{-1}(y)$$ holds, where $$x\in SpecA_{i,\alpha}=
\varphi_{i,\alpha}(U_{i,\alpha});$$ $$y\in
SpecA_{j,\beta}=\varphi_{j,\beta}(U_{j,\beta}).$$

We say $$x\thicksim_{\Sigma} y$$ if there are affine charts
$(U_{i,\alpha},\varphi _{i,\alpha }),(U_{j,\beta},\varphi _{j,\beta })\in
\Delta_{X}$ such that $$\sigma _{i,\alpha}^{-1}(x)=\sigma_{j,\beta
}^{-1}(y)$$ holds, where $$x\in SpecA_{i,\alpha}= \varphi_{i,\alpha}(U_{i,\alpha}), \\
y\in SpecA_{j,\beta}=\varphi_{j,\beta}(U_{j,\beta}),$$ and
$$\sigma
_{i,\alpha}:SpecB_{\alpha}\rightarrow SpecA_{i,\alpha}\text{ and
}\sigma_{j,\beta }:SpecB_{\beta}\rightarrow SpecA_{j,\beta}$$ are the scheme
isomorphisms induced from the ring isomorphisms $$\delta
_{i,\alpha}:A_{i,\alpha}\cong B_{\alpha}\text{ and }\delta_{j,\beta
}:A_{j,\beta}\cong B_{\beta}$$ respectively.

$(iii)$ Let $\Sigma^{*}_{X}$ be the quotient of the set $\Sigma^{*}$ by
$\thicksim_{X}$, and let
$$\pi_{X}:\Sigma^{*}\longrightarrow \Sigma^{*}_{X}$$ be the canonical map.
We have got schemes $\Sigma^{*}$  and $\Sigma^{*}_{X}$ in an evident
manner. It is clear that $\pi_{X}$ is a morphism of the schemes.

Prove that there is an isomorphism $\rho_{X}$ from the scheme
$\Sigma^{*}_{X}$ onto the scheme $X$.

In deed, we have a mapping $$\rho:\Sigma^{*} \longrightarrow X$$ given by
$$z\longmapsto \varphi_{i,\alpha}^{-1}(z),$$ where $(U_{i,\alpha},\varphi
_{i,\alpha })\in\Delta_{X}$ such that $$z\in SpecA_{i,\alpha}=\varphi_{i,\alpha
}(U_{i,\alpha}).$$ Then we have a mapping $$\rho_{X}:\Sigma^{*}_{X} \longrightarrow X,
\pi_{X}(z)\longmapsto \rho (z).$$

Evidently, $\rho_{X}$ is a surjection. From the definition for $\thicksim_{X}$,
it is  seen that $\rho_{X}$ is an injection. Hence, $\rho_{X}$ is a
homeomorphism from the space $\Sigma^{*}_{X}$ onto the space $X$. By
the construction, it is seen that $\rho_{X}$ is an isomorphism of the
schemes.

Similarly, let $\Sigma_{Y}$ be the quotient of the set $\Sigma$ by
$\thicksim_{Y}$, and let
$$\pi_{Y}:\Sigma\longrightarrow \Sigma_{Y}$$ be the canonical map.

Then $\Sigma$ and $\Sigma_{Y}$ are schemes, and $\pi_{Y}$ is a scheme
morphism. There is an isomorphism $\rho_{Y}$ from the scheme
$\Sigma_{Y}$ onto the scheme $Y$.

Let $\Sigma^{*}_{\Sigma}$ be the quotient of the set $\Sigma^{*}$ by
$\thicksim_{\Sigma}$, and let
$$\pi_{\Sigma}:\Sigma^{*}\longrightarrow \Sigma^{*}_{\Sigma}$$ be the canonical map.

Then $\Sigma^{*}_{\Sigma}$ is a scheme and $\pi_{\Sigma}$ is a morphism.
There is an isomorphism $\rho_{\Sigma}$ from the scheme
$\Sigma^{*}_{\Sigma}$ onto the scheme $\Sigma$.

$(iv)$ Take any $x,y\in \Sigma^{*}$. We prove
$$\rho_{X}\circ\pi_{X}(x)=\rho_{X}\circ\pi_{X}(y)$$ if and only if
$$\rho_{Y}\circ\pi_{Y}\circ\rho_{\Sigma}\circ\pi_{\Sigma}(x)=\rho_{Y}\circ
\pi_{Y}\circ\rho_{\Sigma}\circ\pi_{\Sigma}(y).$$

In deed, let $$\rho_{X}\circ\pi_{X}(x)=\rho_{X}\circ\pi_{X}(y).$$ We have
$$\varphi _{i,\alpha}^{-1}(x)= \varphi _{j,\beta }^{-1}(y)$$ for some affine charts
$$(U_{i,\alpha},\varphi _{i,\alpha }), (U_{j,\beta},\varphi _{j,\beta })\in
\Delta_{X}$$ such that $$x\in
SpecA_{i,\alpha}=\varphi_{i,\alpha}(U_{i,\alpha});$$ $$y\in
SpecA_{j,\beta}=\varphi_{j,\beta}(U_{j,\beta}).$$

As $$\mathbb{A}_{0}\left( X,\mathcal{O}_{X}\right) \unlhd\mathbb{A}_{0}\left(
Y, \mathcal{O}_{Y}\right) ,$$ we have affine charts $(V_{i,\alpha},\psi _{i,\alpha
})$ and $ (V_{j,\beta},\psi_{j,\beta })$ contained in
$\mathcal{A}(\Gamma_{Y})$ such that
$$\psi^{-1}_{i,\alpha}\circ\sigma
_{i,\alpha}^{-1}(x)=\psi^{-1}_{j,\beta }\circ\sigma_{j,\beta }^{-1}(y)$$
where
$$\psi _{i,\alpha }(V_{i,\alpha})=SpecB_{i,\alpha},\, \psi_{j,\beta
}(V_{j,\beta})=SpecB_{j,\beta}$$ and
$$\sigma
_{i,\alpha}:SpecB_{i,\alpha}\rightarrow SpecA_{i,\alpha}\text{ and
}\sigma_{j,\beta }:SpecB_{j,\beta}\rightarrow SpecA_{j,\beta}$$
are the isomorphisms induced from
 the ring isomorphisms
 $$\delta
_{i,\alpha}:A_{i,\alpha}\cong B_{i,\alpha}\text{ and }\delta_{j,\beta
}:A_{j,\beta}\cong B_{j,\beta}$$ respectively.

Hence, we have
$$\rho_{Y}\circ\pi_{Y}\circ\rho_{\Sigma}\circ\pi_{\Sigma}(x)=\rho_{Y}\circ\pi_{Y}
\circ\rho_{\Sigma}\circ\pi_{\Sigma}(y).$$

In a similar manner, it is seen  that the converse is true.

$(v)$ Define a map $\tau:X\longrightarrow Y$ by
$$\tau(\rho_{X}\circ\pi_{X}(z))=\rho_{Y}\circ\pi_{Y}\circ\rho_{\Sigma}\circ\pi_{\Sigma}(z)$$
for every $z\in \Sigma^{*}$.

By $(iv)$ it is seen that $\tau$ is well-defined. It is immediate that $\tau$ is the
desired scheme isomorphism from $\left( X,\mathcal{O}_{X}\right) $ onto $
\left( Y,\mathcal{O}_{Y}\right)$.

This completes the proof.
\end{proof}

\end{document}